\documentclass[12pt,leqno,fleqn]{amsart}  
\usepackage{amsmath,amstext,amsthm,amssymb,amsxtra}
\usepackage[top=1.5in, bottom=1.5in, left=1.25in, right=1.25in]	{geometry}
\usepackage{txfonts} % pxfonts txfonts 
\usepackage[T1]{fontenc}
\usepackage{lmodern}

 \usepackage{euler}   % better than the option below
\usepackage{pdfsync}

\usepackage{float} %To specify placement of figures

\usepackage{tikz}

\usepackage{mathtools}
\mathtoolsset{showonlyrefs,showmanualtags}

\usepackage{hyperref} %,hypertexnames=false,colorlinks,[pagebackref]
\hypersetup{
    colorlinks=true,       % false: boxed links; true: colored links
    linkcolor=blue,          % color of internal links
    citecolor=magenta,        % color of links to bibliography
    filecolor=magenta,      % color of file links
    urlcolor=cyan           % color of external links
}

\usepackage[msc-links]{amsrefs}  
\newcommand{\norm}[1]{\ensuremath{\left\|#1\right\|}}
\newcommand{\abs}[1]{\ensuremath{\left\vert#1\right\vert}}
\newcommand{\ip}[2]{\ensuremath{\left\langle#1,#2\right\rangle}}

\newcommand{\R}{\mathbb{R}}

\newcommand{\avg}[1]{\langle #1 \rangle}

\newcommand{\eps}{\varepsilon}

\newcommand{\unit}{1\!\!1}

\newcommand{\esum}{K_{\eps}}
{\end{list}}

\numberwithin{equation}{section}

\newtheorem{thm}{Theorem}[section]
\newtheorem{lm}[thm]{Lemma}

\newtheorem{prop}[thm]{Proposition}
\newtheorem*{prop*}{Proposition}

\theoremstyle{remark}

\newtheorem*{rem*}{Remark}

 \theoremstyle{remark}
 
\numberwithin{equation}{section}

\title[Weak--Type Estimates Involving $A_\infty$ Maximal Functions] 
  {Borderline Weak--Type Estimates for Sparse Bilinear Forms
  Involving $A_\infty$ Maximal Functions}

\author{Rob Rahm}
\address{Texas A\&M Mathematics}
\email{robrahm@math.tamu.edu}

\begin{document}

\begin{abstract}
For any operator $T$ whose bilinear form can be dominated by a sparse bilinear form, 
we prove that $T$ is bounded as a map from $L^1(\widetilde{M}w)$ into weak--$L^1(w)$. Our
main innovation is that $\widetilde{M}$ is a maximal function defined by directly using the 
local $A_\infty$ 
characteristic of the weight (rather than Orlicz norms). Prior results are due to 
Coifman\&Fefferman, P\'{e}rez, Hyt\"onen\&P\'erez, and Domingo-Salazar\&Lacey\&Rey. 
%As 
%we discuss, but do not prove, the maximal functions we use seem to be on the order of 
%$M_{L(\textnormal{log log} L) (\textnormal{log log log} L) (\textnormal{log log log log} L)^{1+\epsilon}}$.% this is smaller 
%than the $M_{L \log\log L (\log\log\log L)^{1+\eps}}$ in \cite{DomLacRey2016}.
\end{abstract}

\maketitle
\section{Introduction}\label{S:intro}
We study weighted endpoint estimates for 
those operators whose bilinear form has a sparse domination. Our 
estimates are in the spirit of Fefferman--Stein \cite{FefSte1971}; in 
particular, for an operator $T$, a function $f$ and a non--negative weight $w$ 
(by non--negative weight we mean a non--negative locally integrable function)
we will prove:
\begin{align}\label{E:intMain}
\lambda w(\{Tf>\lambda\})
\lesssim \int_{\R^d}\abs{f(y)}M_{\varepsilon} w(y)dy, 
\end{align}
where $M_{\varepsilon}$ is a certain maximal function that is pointwise larger 
than the Hardy--Littlewood maximal function. We take an ``entropy bump'' 
point of view -- which is our main innovation -- and define 
$M_{\varepsilon}$ in terms of these entropy bumps. This innovation makes explicit use of 
the important $A_\infty$ characteristic and using this framework we can 
we can prove results of, for example, Hyt\"onen-P\'erez \cite{HytPer2015}.

We now prepare to state our main results. Recall that if $\mathcal{D}$ is a 
dyadic lattice,  
a subset $\mathcal{S}$ of $\mathcal{D}$ is said to be $t$--sparse (or simply ``sparse'') 
($0<t<1$) if every $Q\in\mathcal{S}$ there is a subset 
$E_Q\subset Q$ such that $\abs{E_Q}>t\abs{Q}$ and the sets 
$\{E_Q:Q\in\mathcal{S}\}$ are pairwise disjoint. 

When we say that $T$ has a bilinear sparse domination we mean that 
for all bounded and compactly supported functions $f_1,f_2$ there are 
$\mathcal{K} < \infty$ sparse sets such that there holds:
\begin{align}\label{D:sparseDom}
\abs{\ip{Tf_1}{f_2}}
\lesssim \sum_{k=1}^{\mathcal{K}}\sum_{Q\in\mathcal{S}_k}
  \abs{Q}\avg{f_1}_{Q}\avg{f_2}_{Q},
\end{align}
where $\avg{f}_{Q}:=\frac{1}{\abs{Q}}\int_{Q}\abs{f}$ (note the presence 
of the absolute value inside the integral). The exact value $\mathcal{K}$ isn't important for 
what we are doing, and is different for various operators; it should be pointed out that 
$\mathcal{K}$ depends on $T$ but not $f_1$ nor $f_2$. 

The class of operators that have such a sparse domination is vast. Any Calder\'on--Zygmund 
operator has such a domination, for example (\cite{Ler2013}). See also \cites{ConEtAl2016,ConRey2016,BerFreyPet2017}
and the references therein, for example. Thus, the theorem covers many operators from harmonic analysis.

Let $\eps:[1,\infty]\to[1,\infty]$ be an 
increasing function with $\esum:=\sum_{k=-1}^{\infty}\eps(2^{2^k})^{-1}<\infty$. (The example you should 
keep in mind is essentially $\eps(t) = (\log\log t)(\log \log \log t)^{1+\eps}$) For a 
cube $Q$ and a weight $w(x)\geq 0$ define:
\begin{align*}
\rho_w(Q):=\frac{1}{w(Q)}\int_{Q}M(\unit_{Q}w)(x)dx,
\end{align*}
where $M$ is the usual Hardy--Littlewood maximal function and 
$w(Q):=\int_{Q}w(y)dy$. For a collection $\mathcal{S}$ of cubes (e.g. a 
dyadic lattice or a sparse subset of a dyadic lattice) define the following
maximal function:
\begin{align*}
M_\eps w
:=\sup_{Q\in\mathcal{S}}\unit_{Q}\avg{w}_{Q}\log{\rho_w(Q))}
  \eps(\rho_w(Q)).
\end{align*}
These are our main theorems:
The first is a result of Hyt\"onen--P\'erez 
\cite{HytPer2013}*{Corollary 1.4}:
\begin{thm}\label{T:cor}
Let $T$ be an operator that has a sparse bilinear domination as in 
\eqref{D:sparseDom} and let $w$ be 
an $A_1$ weight. Then we have the following quantitative estimate:
\begin{align*}
\norm{T:L^1(w)\to L^{1,\infty}(w)}\lesssim
[w]_{A_1}\log{(e+[w]_{A_\infty})}, 
\end{align*}
where 
\begin{align*}
[w]_{A_1} 
:= \sup_{Q \textnormal{ a cube}} \frac{M(\unit_{Q}w)(x)}{w(x)}
\hspace{.1in}
\textnormal{and}
\hspace{.1in}
[w]_{A_\infty} 
:= \sup_{Q \textnormal{ a cube}}\frac{1}{w(Q)}\int_{Q}M(\unit_{Q}w)(x)dx.
\end{align*}
\end{thm}

The main theorem (and the new one) is: 

\begin{thm}\label{T:main}
Let $T$ be an operator that has a sparse bilinear domination as in 
\eqref{D:sparseDom} and let $\eps$ be a function as above. Then for 
any weight $w(x)\geq 0$ we have:
\begin{align*}
\norm{T:L^1(M_\eps w)\to L^{1,\infty}(w)}\lesssim\esum.
\end{align*}
\end{thm}

The paper is organized as follows. In the next section, we discuss the 
main result and give some motivation. Following that, in Section \ref{S:bgap} we 
give some background information and preliminary information and then in 
Section \ref{S:pomt} we prove Theorems \ref{T:main} and 
\ref{T:cor}. 
%In Section \ref{S:fio} we prove Theorem \ref{T:fio} and finally in 
%Section \ref{S:concl} we give some concluding remarks. 

\subsection*{Acknowledgment} I am grateful to a very thorough anonymous referee who not only 
caught many typos, mistakes, etc, but also gave constructive suggestions to 
make the exposition better. Of course, any missteps are my fault.

\section{Discussion of Main Results and Previous Results}\label{S:domr}
For the remainder of the paper, the function $\log t$ is the function 
that satisfies $2^{\log t} = 2+t$. That is, the $\log$ we're using here 
is really $\log t = \log_{2}(2+t)$. 

One would like to replace $M_{\varepsilon}$ in \eqref{E:intMain} with 
the smaller Hardy--Littlewood maximal function $M$. However, this is 
not possible; see for example \cites{Reg2011,RegThi2012,HoaMoen2016}. It is of 
interest then to determine the smallest maximal function for which 
\eqref{E:intMain} holds. 

Observe that one way to write $\avg{w}_{Q}$ is $\norm{w}_{L^1(\frac{dx}{\abs{Q}})}$.
Thus to make $M$ slightly larger, we can choose a norm that is slightly larger 
than the normalized $L^1$ norm. A common approach has been to use Orlicz norms.
That is, given a positive non--decreasing function $\Phi$ define:
\begin{align*}
\norm{w}_{Q,\Phi}
:=\inf\{\lambda>0:\frac{1}{\abs{Q}}
  \int_{Q}\Phi(\frac{w(x)}{\lambda})dx\leq 1\}.
\end{align*}

When $\Phi(t)=t^{r}$, then $\norm{w}_{Q,\Phi}=\norm{w}_{L^r(\frac{dx}{\abs{Q}})}$
which is bigger than normalized $L^1$ norm for $1<r$. Maximal functions created 
from these ``power bumps'' were studied in \cite{CoiFef1974}. In 1994, 
P\'erez shows that for singular integral operators (in fact maximal truncations),
 \eqref{E:intMain} holds when $\widetilde{M}$ is the maximal function 
based on $\Phi(t)=t(\log{t})^{1+\eps}$ and this was result was recently 
quantified (in terms of $\eps$) by Hyt\"onen--P\'erez \cites{Per1994,HytPer2015}.
The best known result so far is due to 
Domingo-Salazar, Lacey, and Rey \cite{DomLacRey2016} where 
$\Phi(t)=(\log{\log{t}})(\log{\log{\log{t}}})^{1+\eps}$.  See the papers
listed in the references for more detailed information about these maximal 
functions and Orlicz norms.

In this paper, we take a slightly different approach and use the 
so--called ``entropy bumps'' introduced by Treil--Volberg \cite{TreVol2016}.
More precisely we consider an increasing function $\eps:[1,\infty]\to[1,\infty]$ 
that is just barely summable in the sense that $\esum:=\sum_{k\geq 1}\eps(2^{2^k})^{-1}
<\infty$. For a cube $Q$ we define:
\begin{align*}
\norm{w}_{Q,\rho\eps(\rho)}
:=\avg{w}_{Q}\rho_w(Q)\eps(\rho_{w}(Q)).
\end{align*}
In \cite{TreVol2016} it is shown that  
$t\mapsto \Phi(t)/(t\log{t})$ is increasing for sufficiently large $t$
and $\int^{\infty}\frac{dt}{\Phi(t)}<\infty$, 
then there is a function $\eps$ as above with $\norm{w}_{Q,\rho\eps(\rho)}
\leq\norm{w}_{Q,\Phi}$.

We will use entropy norms that are smaller than the entropy norms defined above.
In particular we will use the following:
\begin{align*}
\norm{w}_{Q,(\log{\rho})\eps(\rho)}
=\avg{w}_{Q}\log{(\rho_w(Q))}\eps(\rho_w(Q)).
\end{align*}

The result mentioned above in \cite{TreVol2016} does not imply that these entropy 
bumps are smaller than the ones in \cite{DomLacRey2016}, and, in fact, the 
Orlicz and entropy maximal functions are probably not directly comparable. One advantage 
of the entropy bumps is that they explicitly use $A_\infty$ data on the 
weight $w$ and it has become clear that this data is important. An advantage of 
the Orlicz norm is of a qualitative nature: the results in \cite{DomLacRey2016} require 
just a bit more than local $L\log\log L$ integrability of the weight, while the 
entropy bumps require local $L\log L$ integrablilty (since $\rho_w(Q) \simeq 
\norm{w}_{L\log L(Q)}$).

The proof(s) in this paper are modifications of the proofs in 
\cites{DomLacRey2016, CulDipOu2016} to the present setting.

\section{Background Information and Preliminaries}\label{S:bgap}

In the proof of the theorems, we will need the collections to satisfy the 
following stronger condition: for every $Q\in\mathcal{S}$ there holds:
\begin{align}\label{E:tradSp}
\abs{\cup_{Q'\in\mathcal{S}:Q'\subsetneq Q}Q'}
\leq \frac{1}{4}\abs{Q}.
\end{align}
The following lemma says that every sparse collection is a finite union of 
sparse collections that satisfy this stronger condition.

\begin{lm}\label{L:stronger}
Let $\mathcal{S}$ be $2^{-l}$--sparse for some $l\in\mathbb{N}$ that 
contains a "top cube" (i.e. a cube that contains all other 
cubes in $\mathcal{S}$). 
Then $\mathcal{S}=\cup_{i=1}^{2^{l+2}}\mathcal{S}^i$ 
where each $\mathcal{S}^i$ satisfies the stronger condition 
\eqref{E:tradSp}.
\end{lm}

\begin{proof}
The sparse condition implies the ``Carleson'' condition (recall that if 
$Q'\in\mathcal{S}$ then $\abs{E_{Q'}}>2^{-l}\abs{Q'}$): 
for every $Q\in\mathcal{S}$ we have:
\begin{align*}
\sum_{Q'\in\mathcal{S}:Q'\subsetneq Q}\abs{Q'}
\leq 2^l\sum_{Q'\in\mathcal{S}:Q'\subsetneq Q}\abs{E_{Q'}}
\leq 2^l\abs{Q}.
\end{align*}
For a $Q\in\mathcal{S}$ define $\mathcal{S}_1(Q)$ be those cubes in $\mathcal{S}$ 
that are contained in $Q$ and are maximal in $Q$ (i.e. $Q'$ is in $\mathcal{S}_1(Q)$ if 
$Q'\in \mathcal{S}$, $Q'\subset Q$ and there is no cube $Q''\in\mathcal{S}$ with $Q'\subsetneq Q'' \subsetneq Q)$. 
Inductively, define $\mathcal{S}_{k+1}(Q)$ to be those cubes that are in $\mathcal{S}$, are contained in 
some $R\in\mathcal{S}_{k}(Q)$ and are maximal in that cube (i.e. $Q'\in\mathcal{S}_{k+1}(Q)$ if 
$Q'\in\mathcal{S}$, $Q'\subset R$ for some $R\in\mathcal{S}_{k}(Q)$ and there is no $Q''\in\mathcal{S}$ with 
$Q' \subsetneq Q'' \subsetneq R$). Informally, $\mathcal{S}_{k}(Q)$ are the cubes in $\mathcal{S}$ 
that are $k$--generations down in $\mathcal{S}$ from $Q$ (or the $k$--children of $Q$ in $\mathcal{S}$). 

We claim that 
$\abs{\cup_{Q'\in\mathcal{S}_{2^{l+2}}(Q)}Q'}\leq\frac{1}{4}\abs{Q}$. Indeed, 
suppose not; then we would have:
\begin{align*}
\sum_{Q'\in\mathcal{S}:Q'\subsetneq Q}\abs{Q}
>\sum_{k=1}^{2^{l+2}}\sum_{Q'\in\mathcal{S}_{k}(Q)}\abs{Q'}
>\sum_{k=1}^{2^{l+2}}\frac{1}{4}\abs{Q}
= \frac{2^{l+2}}{4}\abs{Q}
= 2^l\abs{Q},
\end{align*}
which violates the Carleson condition. It is now easy to see how to 
separate $\mathcal{S}$ into $2^{l+2}$ sub--collections: let $Q_0$ be the top 
cube in $\mathcal{S}$. For $k=1,\ldots, 2^{l+2}$ let 
\begin{align*}
\mathcal{S}_{k}=
\cup_{n\geq 0}\cup_{Q\in\mathcal{S}_{2^{l+2}n+k}(Q_0)}Q. 
\end{align*}
Thus, for each 
$Q\in\mathcal{S}_k$, the cubes one generation down in $\mathcal{S}_k$ are 
$2^{l+2}$ generations down in $\mathcal{S}$ and so we have the stronger
sparse condition:
\begin{align*}
\abs{\cup_{Q'\in\mathcal{S}_k:Q'\subsetneq Q}Q'}
\leq \frac{1}{4}\abs{Q},
\end{align*}
as desired.
\end{proof}

We now have a variant of the classic Fefferman--Stein Inequality (see also \cite{CruMarPer2011}.)
(Observe that the lemma and its proof are classical and well--known; however, since this maximal 
function only takes a $\sup$ over a subset of $\mathcal{S}$, we can't quote 
earlier results). 
Let $\mathcal{S}$ be a subset of some dyadic lattice and 
define $M_\mathcal{S}f:=\sup_{Q\in\mathcal{S}}\unit_{Q}\avg{f}_{Q}$.
\begin{lm}\label{L:fefst}
For every $f$ and $\lambda > 0$ we have:
\begin{align*}
\lambda w(\{M_\mathcal{S}f > \lambda\})
\leq \int_{\R^d}\abs{f(y)}M_\mathcal{S}w(y)dy.
\end{align*}
\end{lm}
\begin{proof}
For $\lambda >0$, let $\Omega_\lambda$ be the maximal cubes in 
$\mathcal{S}$ with $\avg{f}_{Q}>\lambda$. Then using the fact that 
the cubes in $\Omega_\lambda$ are pairwise disjoint, there holds
\begin{align*}
\lambda w(\{M_\mathcal{S}f>\lambda\})
%&=\sum_{Q\in\mathcal{H}}\frac{4w(G)^{-1}}{4w(G)^{-1}}w(Q)
&\leq \sum_{Q\in\Omega_\lambda}\avg{f}_{Q}w(Q)
\\&= \sum_{Q\in\Omega_\lambda}
  \int_{Q}\abs{f(y)}dy\frac{w(Q)}{\abs{Q}}
%\\&= \frac{1}{4w(G)^{-1}}\sum_{Q\in\mathcal{H}}
%  \int_{Q}\abs{f(y)}\avg{w}_{Q}dy
%\\&= \frac{1}{4w(G)^{-1}}\sum_{Q\in\mathcal{H}}
%  \int_{Q}\abs{f(y)}Mw(y)dy
\\&\leq \int_{\R^d}\abs{f(y)}M_\mathcal{S}w(y)dy,
\end{align*}
as desired.
\end{proof}

There are many ways to define the $[w]_{A_\infty}$ characteristic of a 
weight. The one we use -- and the one that seems most useful and 
popular -- is the one of Wilson \cite{Wil1987}; see also 
\cite{HytPer2015} for more information. 
For a dyadic lattice $\mathcal{D}$, let $M_{\mathcal{D}}$ be the dyadic 
maximal function. For $Q\in\mathcal{D}$ define:
\begin{align*}
\rho_w(Q)
:=\frac{1}{w(Q)}\int_{Q}M_\mathcal{D}(w\unit_{Q})(x)dx
\end{align*}
The following is \cite{HytPer2013}*{Lemma 6.6}:
\begin{lm}\label{L:ainfest}
For a cube $Q\in\mathcal{D}$ and a subset $E\subset Q$ we have
\begin{align*}
w(E)
\lesssim w(Q)\frac{\rho_w(Q)}{\log{\frac{\abs{E}}{\abs{Q}}}}.
\end{align*}
\end{lm}

\section{Proofs of Theorems \ref{T:cor} and \ref{T:main}}\label{S:pomt}
Recall that we are dealing with operators $T$ whose bilinear form has a 
sparse domination. That is, there are $\mathcal{K}<\infty$ sparse sets such 
that 
\begin{align*}
\abs{\ip{Tf_1}{f_2}}
\lesssim \sum_{k=1}^{\mathcal{K}}
         \sum_{Q\in\mathcal{S}_k}\abs{Q}\avg{\abs{f_1}}_{Q}\avg{\abs{f_2}}_{Q}.
\end{align*}
The choice of sparse sets depend on $f_1$, $f_2$ and, $T$, but
$\mathcal{K}$ and the implied constant are independent of $f_1$ and $f_2$, but depend on $T$ and the geometry of the 
space (i.e. $\R^{d}$). 
%The implied constants in our estimate are allowed to depend on $\mathcal{K}$, 
%and so we may assume that 
%\begin{align*}
%\abs{\ip{Tf_1}{f_2}} \lesssim 
%\sum_{Q\in\mathcal{S}}\abs{Q}\avg{\abs{f_1}}_{Q}\avg{\abs{f_2}}_{Q}.
%\end{align*}
%In other words, we assume there is only one sparse collection in the %domination 
%(and recall, when we say sparse now we mean as in \eqref{E:tradSp}). 

The rest of the section is used to prove the following proposition, but before 
we prove the proposition, we show how Theorems \ref{T:cor} and 
\ref{T:main} are corollaries of the proposition. (It might 
seem strange that there is a $2^{2^r}$ and $2^r$ in the 
proposition below, instead of $r$ and $\log r$, but when this
proposition is applied, it is applied with $2^{2^r}$ and 
$2^r$).

\begin{prop}\label{P:mp}
Let $w$ be a weight (i.e. non--negative locally integrable function). Then for any 
sparse collection $\mathcal{S}$ with 
$\sup_{Q\in\mathcal{S}}\rho_{w}(Q)\leq 2^{2^r}$ for 
some $r\in\mathbb{N}$ and for all non--negative locally integrable functions $f$ there holds
\begin{align*}
\sup_{\substack{G\subset\R^d\\ 0<w(G)<\infty}}
  \hspace{.01in}
  \inf_{\substack{G'\subset G\\
  w(G)\leq2w(G')}} \sum_{Q\in\mathcal{S}}\abs{Q}\avg{f}_{Q}\avg{w\unit_{G'}}_{Q}
\lesssim 2^r\int_{\R^{d}}\abs{f(y)}M_{\mathcal{S}}w(y)dy. 
\end{align*}
where $M_\mathcal{S}$ is the maximal function restricted to cubes in 
$\mathcal{S}$: 
\begin{align*}
M_{\mathcal{S}}f(x):=\sup_{Q\in\mathcal{S}}\avg{\abs{f}}_{Q}\unit_{Q}(x).
\end{align*}
\end{prop}

\subsection{Deducing the Main Theorems}
First, recall that the weighted weak $L^1$ norm can be computed as follows:
\begin{align*}
\norm{g}_{L^{1,\infty}(\mu)} 
= \sup_{\substack{G\subset\R^d\\ 0<\mu(G)<\infty}}
  \hspace{.01in}
  \inf_{\substack{G'\subset G\\
  \mu(G)\leq2\mu(G')}}\abs{\ip{g}{\mu\unit_{G'}}}.
\end{align*}
Thus, for a weight $w$ and an operator $T$: 
\begin{align}\label{E:wl1}
\norm{Tf}_{L^{1,\infty}(w)} 
= \sup_{\substack{G\subset\R^d\\ 0<w(G)<\infty}}
  \hspace{.01in}
  \inf_{\substack{G'\subset G\\
  w(G)\leq2w(G')}}\abs{\ip{Tf}{w\unit_{G'}}}.
\end{align}
Therefore if $T$ has a bilinear domination as defined in Section \ref{S:intro}: 
\begin{align*}
\abs{\ip{Tf_f}{f_2}}
\lesssim \sum_{k=1}^{\mathcal{K}}\sum_{Q\in\mathcal{S}_k}
    \abs{Q}\avg{f_1}_{Q}\avg{f_2}_{Q},
\end{align*}
then via standard reductions, and \eqref{E:wl1}, it is enough to estimate: 
\begin{align*}
\sup_{\substack{G\subset\R^d\\ 0<w(G)<\infty}}
  \hspace{.01in}
  \inf_{\substack{G'\subset G\\
  w(G)\leq2w(G')}} \sum_{Q\in\mathcal{S}}\abs{Q}\avg{f}_{Q}\avg{w\unit_{G'}}_{Q},
\end{align*}
uniformly over all sparse sets $\mathcal{S}$ and non--negative, compactly supported 
functions $f$. Furthermore, by (for example) the monotone convergence theorem, 
we may assume that there is a finite number of cubes in $\mathcal{S}$ (to avoid 
any convergence issues in the estimates below) and we may assume there is a 
"top cube", that is a cube in $\mathcal{S}$ that contains all other cubes in 
$\mathcal{S}$ (so we can apply Lemma \ref{L:stronger}).

\begin{proof}[Proof of Theorem \ref{T:cor}]
To deduce Theorem \ref{T:cor} from Proposition \ref{P:mp} observe that if 
$w\in A_\infty$ and $2^{2^{r-1}}<[w]_{A_\infty} \leq 2^{2^r}$, then we may take any sparse 
collection 
to be the collection in Proposition \ref{P:mp}. Since $2^r\simeq \log [w]_{A_\infty}$ and 
$Mw(y) \leq [w]_{A_1}w(y)$, Proposition \ref{P:mp} and the above reductions 
assert that: 
\begin{align*}
\norm{T}_{L^{1,\infty}(w)}
\lesssim 
2^r\int_{\R^{d}}\abs{f(y)}M_{\mathcal{S}}w(y)dy 
\lesssim [w]_{A_1}\log [w]_{A_\infty}\int_{\R^{d}}\abs{f(y)}w(y)dy,
\end{align*}
which is Theorem \ref{T:cor}.
\end{proof}

Now we deduce Theorem \ref{T:main} from Proposition \ref{P:mp}.
\begin{proof}[Proof of Theorem \ref{T:main}]
We estimate: 
\begin{align*}
\sup_{\substack{G\subset\R^d\\ 0<w(G)<\infty}}
  \hspace{.01in}
  \inf_{\substack{G'\subset G\\
  w(G)\leq2w(G')}} \sum_{Q\in\mathcal{S}}\abs{Q}\avg{f}_{Q}\avg{w\unit_{G'}}_{Q},
\end{align*}
uniformly over all sparse collections $\mathcal{S}$ and non--negative locally 
integrable functions $f$. Let $\mathcal{S}_r$ be those cubes in 
$\mathcal{S}$ that satisfy $2^{2^{r-1}}<\rho_w(Q)\leq 2^{2^r}$. 
Then by Proposition 
\ref{P:mp} there holds: 
\begin{align}\label{E:mp}
\sup_{\substack{G\subset\R^d\\ 0<w(G)<\infty}}
  \hspace{.01in}
  \inf_{\substack{G'\subset G\\
  w(G)\leq2w(G')}} \sum_{Q\in\mathcal{S}_r}\abs{Q}\avg{f}_{Q}\avg{w\unit_{G'}}_{Q}
\lesssim 2^r\int_{\R^{d}}\abs{f(y)}M_{\mathcal{S}_r}w(y)dy. 
\end{align}
We now have a critical observation. The maximal function 
$M_{\mathcal{S}_r}$ only considers those cubes in $\mathcal{S}_r$. This 
means that for those cubes, $\log \rho_w(Q)\simeq 2^r$ and simiarly, 
using the fact that $\eps$ is increasing, $\eps(2^{2^{r-1}})<
\eps(\rho_w(Q))$. Therefore we may estimate: 
\begin{align*}
2^rM_{\mathcal{S}_r}w(y) 
&= \frac{1}{\eps(2^{2^{r-1}})}\sup_{Q\in\mathcal{S}_r} \avg{w}_{Q}2^r \eps({2^{2^{r-1}}})\unit_{Q}(y)
\\&\simeq \frac{1}{\eps(2^{2^{r-1}})}\sup_{Q\in\mathcal{S}_r} \avg{w}_{Q}
\log(\rho_w(Q))\eps({2^{2^{r-1}}})\unit_{Q}(y)
\\&\leq \frac{1}{\eps(2^{r-1})}\sup_{Q\in\mathcal{S}_r}
    \log(\rho_{w}(Q))\eps(\rho_{w}(Q)) \unit_{Q}(y)
\\&\leq \frac{1}{\eps(2^{r-1})}M_{\eps}w(y).
\end{align*}
Thus, the estimate in \eqref{E:mp} can be continued as 
\begin{align}\label{E:mp1}
\sup_{\substack{G\subset\R^d\\ 0<w(G)<\infty}}
  \hspace{.01in}
  \inf_{\substack{G'\subset G\\
  w(G)\leq2w(G')}} \sum_{Q\in\mathcal{S}_r}\abs{Q}\avg{f}_{Q}\avg{w\unit_{G'}}_{Q}
&\lesssim 2^r\int_{\R^{d}}\abs{f(y)}M_{\mathcal{S}_r}w(y)dy
\\&\lesssim \frac{1}{\eps(2^{r-1})}\int_{\R^{d}}\abs{f(y)}M_{\eps}w(y)dy.
\end{align}
Using the summability condition on $\eps$, this estimate can be summed in 
$r$ to the desired estimate in Theorem \ref{T:main}.
\end{proof}

\subsection{Proving Proposition \ref{P:mp}}
We now turn to proving Proposition \ref{P:mp}. By homogeneity, it 
suffices to prove
\begin{align*}
\sup_{f\geq0:\norm{f}_{L^1(M_{\mathcal{S}}w)}=1}
\sup_{\substack{G\subset\R^d\\ 0<w(G)<\infty}}
  \hspace{.01in}
  \inf_{\substack{G'\subset G\\
  w(G)\leq2w(G')}} \sum_{Q\in\mathcal{S}_r}\abs{Q}\avg{f}_{Q}\avg{w\unit_{G'}}_{Q}
\lesssim 2^r,
\end{align*}
where this estimate is uniform over all sparse collections with 
$\sup_{Q\in\mathcal{S}}\rho_w(Q)\leq 2^{2^r}$.

Fix a set $G$ with $0<w(G)<\infty$ and a compactly supported non--negative 
function $f$ with $\norm{f}_{L^1(M_\mathcal{S} w)}=1$. Since $f$ is bounded and 
compactly supported, we may assume that $f$ is supported on a cube $Q_0\in\mathcal{S}$ and that $\avg{{f}}_{Q_0} < 4w(G)^{-1}$. Furthermore, 
we may assume that $Q_0$ contains all cubes in $\mathcal{S}$.

Let $\mathcal{H}$ be the maximal cubes in $\mathcal{S}$ 
with $\avg{{f}}_{Q} > 4w(G)^{-1}$ and set 
$H=\cup_{Q\in\mathcal{H}}Q$. 
By the Fefferman--Stein Inequality (Lemma \ref{L:fefst}) -- we have 
$w(H)\leq \frac{1}{4}w(G)$. Indeed, taking $\lambda = 4w(G)^{-1}$ and 
noting that $\norm{f}_{L^1(M_\mathcal{S} w)}=1$ there holds
\begin{align*}
w(H)
= w(\{M_{\mathcal{S}}f > \lambda\})
\leq \frac{1}{\lambda}\int_{\R^{d}}\abs{f} M_{\mathcal{S}}w 
\leq \frac{1}{\lambda}\int_{\R^{d}}\abs{f} M_{\mathcal{S}}w 
= \frac{w(G)}{ 4}.
\end{align*}

Set $G' = G\cap H^{c}$ and note that:
\begin{align*}
w(G)
= w(G\cap H) + w(G\cap H^{c})
\leq w(H) + w(G')
\leq \frac{1}{4}w(G) + w(G'),
\end{align*}
and so $w(G) \leq 2w(G')$.

We may now assume that the cubes in $\mathcal{S}$ satisfy $\avg{f}_{Q}\leq 
4w(G)^{-1}$. 
If not then $Q$ is either a cube in $\mathcal{H}$ or is contained in a cube 
in $\mathcal{H}$; either way, $Q$ is contained in $H$. But $\unit_{G'}$ is 
zero on $H$ and so $\avg{\unit_{G'}}_{Q}=0$. Thus, we assume that 
$\avg{f}_{Q}\leq 4w(G)^{-1}$. 

For $k\geq -1$, let 
$$\mathcal{S}_{k}:=\{Q\in\mathcal{S}:4^{-k}w(G)^{-1}<\avg{f}_{Q}\leq
2\cdot4^{-k}w(G)^{-1}\}.$$ 
For fixed $k$, let $S_{k}^{0}$ be the maximal cubes 
in $\mathcal{S}_{k}$ and for $j\geq 1$ set $\mathcal{S}_{k}^{j}$ be the 
maximal cubes in $\mathcal{S}_{k}\setminus\cup_{l=0}^{j-1}\mathcal{S}_{k}^{l}$.
For $Q\in\mathcal{S}_{k}^{j}$ let $E_Q=Q\setminus 
\cup_{Q'\in\mathcal{S}_{k}^{j+1}}Q'$. Observe that the sets 
$\{E_Q:Q\in\mathcal{S}_{k}\}$ are pairwise disjoint. The sparsity condition 
\eqref{E:tradSp} implies: 
\begin{lm}\label{L:comp}
$\int_{Q}\abs{f(y)}dy\simeq\int_{E_Q}\abs{f(y)}dy$.
\end{lm}
\begin{proof}[Proof of Lemma \ref{L:comp}]
If $\mathcal{S}_{k}^{l}(Q)$ are those cubes in 
$\mathcal{S}_{k}^{l+1}$ that are contained in $Q$, there holds: 
\begin{align*}
\int_{Q}\abs{f(y)}dy
&=\int_{E_Q}\abs{f(y)}dy + \sum_{Q'\in\mathcal{S}_{k}^{l}(Q)}\abs{Q'}\avg{\abs{f}}_{Q'}
\\& \leq \int_{E_Q}\abs{f(y)}dy + \sum_{Q'\in\mathcal{S}_{k}^{l}(Q)}\abs{Q'}2\cdot4^{-k}w(G)^{-1}.
\end{align*}
Using the definition of $\mathcal{S}_{k}$ we see that $4^{-k}w(G)^{-1} < \avg{\abs{f}}_{Q}$. This 
combined with the condition \eqref{E:tradSp} allows us to continue the estimate above as: 
\begin{align*}
\int_{Q}\abs{f(y)}dy
&=\int_{E_Q}\abs{f(y)}dy + \sum_{Q'\in\mathcal{S}_{k}^{l}(Q)}\abs{Q'}2\cdot4^{-k}w(G)^{-1}
\\& \leq \int_{E_Q}\abs{f(y)}dy  + \frac{1}{2}\abs{Q}\avg{\abs{f}}_{Q}
\\&= \int_{E_Q}\abs{f(y)}dy + \frac{1}{2}\int_{Q}\abs{f(y)}dy, 
\end{align*}
from which we conclude $\int_{Q}\abs{f(y)}dy\simeq\int_{E_Q}\abs{f(y)}dy$ as desired.
\end{proof}

We break the $k$ sum into two pieces
\begin{align*}
\sum_{k=-1}^{10\cdot 2^R}
\sum_{Q\in\mathcal{S}_r}\abs{Q}\avg{f}_{Q}\avg{w\unit_{G'}}_{Q}
+ \sum_{10\cdot 2^r}^{\infty}
\sum_{Q\in\mathcal{S}_r}\abs{Q}\avg{f}_{Q}\avg{w\unit_{G'}}_{Q}
\end{align*}
and show each piece is controlled by $2^r$ and $1$, 
respectively. We handle each piece in the following two subsections. 

\subsection{Main Portion ($k\leq10\cdot 2^r$)}
Observe: 
\begin{align*}
\abs{Q}\avg{w\unit_{G'}}_{Q} 
= w(G'\cap Q).  
\end{align*}
The goal of this subsection is to prove: 
\begin{lm}\label{L:est1}
\begin{align*}
\sum_{k=-1}^{10\cdot 2^r}\sum_{Q\in\mathcal{S}_{r,k}}\avg{f}_{Q}w(G'\cap Q)
\lesssim 2^r.
\end{align*}
\end{lm}
\begin{proof}[Proof of Lemma \ref{L:est1}]
This follows from the following estimates. 

First, using Lemma \ref{L:comp} and the 
pairwise disjointness of the sets $\{E_Q\}$, for fixed $k$ there holds:
\begin{align*}
\sum_{Q\in\mathcal{S}_{k}}\avg{f}_{Q}w(G'\cap Q)
%&=\sum_{Q\in\mathcal{S}_k}\int_{Q}\abs{f(y)}dy \avg{w}_{Q}
\simeq \sum_{Q\in\mathcal{S}_{k}}\int_{E_Q}\abs{f(y)}dy \avg{w}_{Q}
\leq \int_{\R^d}\abs{f(y)}M_{\mathcal{S}_k}w(y),
%\leq 1,
\end{align*}
where $M_{\mathcal{S}_k}$ is the maximal function where the supremum is taken 
over cubes in $\mathcal{S}_k$.
Using the fact that this estimate holds uniformly in $k$, we 
can make the following coarse estimate: 
\begin{align}\label{E:Lest1}
\sum_{k=-1}^{10\cdot 2^r}
\sum_{Q\in\mathcal{S}_{k}}\abs{Q}\avg{f}_{Q}w(G'\cap Q)
%&\leq 10\cdot 2^r \max_{-1\leq k \leq 10\cdot 2^r}
%\sum_{Q\in\mathcal{S}_{r,k}}\abs{Q}\avg{f}_{Q}w(G'\cap Q)
\leq 10\cdot 2^r \int_{\R^d}\abs{f(y)}M_{\mathcal{S}}w(y)dy,
\end{align}
and using the fact that $\norm{f}_{L^1(M_{\mathcal{S}}w)}=1$ this completes 
the proof of Lemma \ref{L:est1}.
\end{proof}

\subsection{The Tail ($k\geq 10\cdot2^r$)} The goal of this subsection is to prove 
\begin{lm}\label{L:est3}
\begin{align*}
\sum_{k=10\cdot 2^r}^{\infty}\sum_{Q\in\mathcal{S}_{r,k}}\avg{f}_{Q}w(G'\cap Q)
\lesssim 1.
\end{align*}
\end{lm}

For a cube $Q$ in 
$\mathcal{S}_{k}^{j}$, let $\mathcal{S}_{k}^{j+t}(Q)$ be the cubes in 
$\mathcal{S}_{k}^{j+t}$ contained in $Q$ and define
$Q_{t}:=\cup_{Q'\in\mathcal{S}_{k}^{j+t}(Q)}Q'$ 
where $t=2^k$. The sparse condition implies that $\abs{Q_t}\leq 
4^{-t}\abs{Q}$ and Lemma \ref{L:ainfest} implies that 
$w(Q_t)\lesssim 2^{2^r}2^{-k}w(Q)$. Note that we may write:
\begin{align*}
Q = Q_{t} \cup (\cup_{l=0}^{t-1}\cup_{Q'\in\mathcal{S}_{k}^{j+l}(Q)}E_Q).
\end{align*}

Concerning the $Q_t$ portion, for $Q$ in $\mathcal{S}_{k}$ and Lemma 
\ref{L:comp}:
\begin{align*}
\avg{f}_{Q}w(G'\cap Q_t)
\leq \frac{2^{2^r}}{2^k}\avg{f}_{Q}w(Q)
=\frac{2^{2^r}}{2^k} \int_{Q}\abs{f(y)}dy \avg{w}_{Q}
\simeq \frac{2^{2^r}}{2^k} \int_{E_Q}\abs{f(y)}dy \avg{w}_{Q}.
\end{align*}
(The "$\leq$" is Lemma \ref{L:ainfest} and the "$\simeq$" is 
Lemma \ref{L:comp}).
Thus for fixed $k$ we have -- using the pairwise disjointness of the 
sets $\{E_Q:Q\in\mathcal{S}_{k}\}$:
\begin{align*}
\sum_{Q\in\mathcal{S}_{k}}\avg{{f}}_{Q}w(G'\cap Q_t)
\lesssim \frac{2^{2^r}}{2^k}\int_{\R^d}\abs{f(y)}Mw(y)
\leq \frac{2^{2^r}}{2^k}.
\end{align*}
This can be summed in $k\geq 10\cdot2^r$ to the desired estimate. 

We must now handle the portion involving $Q\setminus Q_t$. Note that 
for fixed $l$ and $k$, the sets $\{E_{Q'}: Q'\in\mathcal{S}_{k}^{j+l}(Q) \textnormal{ and } 
Q\in\mathcal{S}_{k}^{j}; j\geq 0\}$ are pairwise disjoint. 
Thus for fixed $k$ we have:
\begin{align*}
\sum_{j\geq 0}\sum_{Q\in\mathcal{S}_{k}^j}\sum_{l=0}^{t-1}
  \sum_{\substack{Q'\in\mathcal{S}_{k}^{j+l}:\\Q'\subset Q}}
  \avg{f}_{Q}w(G'\cap E_{Q'})
\leq 4^{-k}w(G)^{-1}t\sum_{j\geq 0}\sum_{Q\in\mathcal{S}_{k}^{j}}
  w(G'\cap \widetilde{E}_Q),
\end{align*}
where the sets $\widetilde{E}_Q$ are pairwise disjoint according to the observation
above. Therefore this 
term is bounded by $4^{-k}tw(G)^{-1}w(G')\leq 2^{-k}$. This can be summed 
in $k\geq 10\cdot2^r$ to the desired estimate.

%%%%%%%%%%%%
%%%References%%%
%%%%%%%%%%%%

%%%%%%%%%%%%
%%%%END%%%%%%
%%%%%%%%%%%%

\begin{bibdiv}
\begin{biblist}
\bib{BerFreyPet2017}{article}{
   author={Bernicot, Fr\'ed\'eric},
   author={Frey, Dorothee},
   author={Petermichl, Stefanie},
   title={Sharp weighted norm estimates beyond Calder\'on-Zygmund theory},
   journal={Anal. PDE},
   volume={9},
   date={2016},
   number={5},
   pages={1079--1113},
   issn={2157-5045},
   review={\MR{3531367}},
   doi={10.2140/apde.2016.9.1079},
}

\bib{CoiFef1974}{article}{
   author={Coifman, R. R.},
   author={Fefferman, C.},
   title={Weighted norm inequalities for maximal functions and singular
   integrals},
   journal={Studia Math.},
   volume={51},
   date={1974},
   pages={241--250},
   issn={0039-3223},
   review={\MR{0358205}},
}


\bib{ConEtAl2016}{article}{
   author={Conde-Alonso, Jos\'{e} M.},
   author={Culiuc, Amalia},
   author={Di Plinio, Francesco},
   author={Ou, Yumeng},
   title={A sparse domination principle for rough singular integrals},
   journal={Anal. PDE},
   volume={10},
   date={2017},
   number={5},
   pages={1255--1284}
}

\bib{ConRey2016}{article}{
   author={Conde-Alonso, Jos\'e M.},
   author={Rey, Guillermo},
   title={A pointwise estimate for positive dyadic shifts and some
   applications},
   journal={Math. Ann.},
   volume={365},
   date={2016},
   number={3-4},
   pages={1111--1135},
}

\bib{Cru2015}{article}{
   author={Cruz-Uribe, David},
   title={Two weight inequalities for fractional integral operators and
   commutators},
   conference={
      title={Advanced courses of mathematical analysis VI},
   },
   book={
      publisher={World Sci. Publ., Hackensack, NJ},
   },
   date={2017},
   pages={25--85},
    eprint={http://arxiv.org/abs/1412.4157}
}


\bib{CruMarPer2011}{book}{
   author={Cruz-Uribe, David V.},
   author={Martell, Jos\'{e} Maria},
   author={P\'{e}rez, Carlos},
   title={Weights, extrapolation and the theory of Rubio de Francia},
   series={Operator Theory: Advances and Applications},
   volume={215},
   publisher={Birkh\"{a}user/Springer Basel AG, Basel},
   date={2011},
   pages={xiv+280},
   isbn={978-3-0348-0071-6},
   review={\MR{2797562}},
   doi={10.1007/978-3-0348-0072-3},
}

\bib{CulDipOu2016}{article}{
    author={Culiuc, Amalia},
    author={Di Plinio, Francesco},
    author={Ou, Yumeng},
    title={Uniform sparse domination of singular integrals via dyadic shifts},
    date={2016},
    eprint={https://arxiv.org/abs/1610.01958}
}

\bib{DomLacRey2016}{article}{
   author={Domingo-Salazar, Carlos},
   author={Lacey, Michael},
   author={Rey, Guillermo},
   title={Borderline weak-type estimates for singular integrals and square
   functions},
   journal={Bull. Lond. Math. Soc.},
   volume={48},
   date={2016},
   number={1},
   pages={63--73}
}

\bib{FefSte1971}{article}{
   author={Fefferman, C.},
   author={Stein, E. M.},
   title={Some maximal inequalities},
   journal={Amer. J. Math.},
   volume={93},
   date={1971},
   pages={107--115},
   issn={0002-9327},
   review={\MR{0284802}},
   doi={10.2307/2373450},
}

% \bib{Hyt}{article}{
%    author={Hyt\"{o}nen, Tuomas P.},
%    title={The $A_2$ Theorem: Remarks
%    and Complements},
%    date={2012},
%    eprint={http://www.arxiv.org/abs/1212.3840},
% }

\bib{HoaMoen2016}{article}{
    author={Hoang, Cong},
    author={Moen, Kabe},
    title={Muckenhoupt-Wheeden conjectures for sparse operators},
    date={2016},
    eprint={https://arxiv.org/abs/1609.03889}
}

\bib{HytPer2013}{article}{
   author={Hyt\"{o}nen, Tuomas},
   author={P\'{e}rez, Carlos},
   title={Sharp weighted bounds involving $A_\infty$},
   journal={Anal. PDE},
   volume={6},
   date={2013},
   number={4},
   pages={777--818},
   issn={2157-5045},
   review={\MR{3092729}},
   doi={10.2140/apde.2013.6.777},
}

\bib{HytPer2015}{article}{
   author={Hyt\"onen, Tuomas},
   author={P\'erez, Carlos},
   title={The $L(\log L)^\epsilon$ endpoint estimate for maximal singular
   integral operators},
   journal={J. Math. Anal. Appl.},
   volume={428},
   date={2015},
   number={1},
   pages={605--626}
}

\bib{Ler}{article}{
   author={Lerner, Andrei K.},
   title={A pointwise estimate for the local sharp maximal function with
   applications to singular integrals},
   journal={Bull. Lond. Math. Soc.},
   volume={42},
   date={2010},
   number={5},
   pages={843--856}
}

\bib{Ler2013}{article}{
   author={Lerner, Andrei K.},
   title={A simple proof of the $A_2$ conjecture},
   journal={Int. Math. Res. Not. IMRN},
   date={2013},
   number={14},
   pages={3159--3170},
   issn={1073-7928},
   review={\MR{3085756}},
}

% \bib{LacSawUri}{article}{
%    author={Lacey, Michael T.},
%    author={Sawyer, Eric T.},
%    author={Uriarte-Tuero, Ignacio},
%    title={Two Weight Inequalities for
%    Discrete Positive Operators},
%    date={2009},
%    eprint={http://arxiv.org/abs/0911.3437},
% }

\bib{Lac2015}{article}{
  author={Lacey, Michael T.},
  title={An Elementary Proof of the $A_2$ Bound},
  date={2015},
  eprint={http://arxiv.org/abs/1501.05818}
}

\bib{LacMen2016}{article}{
  author={Lacey, Michael T.},
  author={Mena, Dar\'io},
  title={The Sparse $T1$ Theorem},
  date={2016},
  eprint={https://arxiv.org/abs/1610.01531}
}

% \bib{LacSpe}{article}{
%     author={Lacey, Michael T.},
%     author={Spencer, Scott},
%     title={On Entropy Bounds for Calder\'{o}n--Zygmund Operators},
%     date={2015},
%     eprint={http://arxiv.org/abs/1504.02888}
% }



\bib{Per1994}{article}{
   author={P\'erez, C.},
   title={Weighted norm inequalities for singular integral operators},
   journal={J. London Math. Soc. (2)},
   volume={49},
   date={1994},
   number={2},
   pages={296--308}
}

\bib{Reg2011}{article}{
   author={Reguera, Maria Carmen},
   title={On Muckenhoupt-Wheeden conjecture},
   journal={Adv. Math.},
   volume={227},
   date={2011},
   number={4},
   pages={1436--1450},
   issn={0001-8708},
   review={\MR{2799801}},
   doi={10.1016/j.aim.2011.03.009},
}

\bib{RegThi2012}{article}{
   author={Reguera, Maria Carmen},
   author={Thiele, Christoph},
   title={The Hilbert transform does not map $L^1(Mw)$ to $L^{1,\infty}(w)$},
   journal={Math. Res. Lett.},
   volume={19},
   date={2012},
   number={1},
   pages={1--7},
   issn={1073-2780},
   review={\MR{2923171}},
   doi={10.4310/MRL.2012.v19.n1.a1},
}

% \bib{Neu}{article}{
%    author={Neugebauer, C. J.},
%    title={Inserting $A_{p}$-weights},
%    journal={Proc. Amer. Math. Soc.},
%    volume={87},
%    date={1983},
%    number={4},
%    pages={644--648}
% }

% \bib{Per}{article}{
%    author={P{\'e}rez, Carlos},
%    title={Two weighted inequalities for potential and fractional type
%    maximal operators},
%    journal={Indiana Univ. Math. J.},
%    volume={43},
%    date={1994},
%    number={2},
%    pages={663--683}
% }

% \bib{Per1}{article}{
%    author={P{\'e}rez, Carlos},
%    title={On sufficient conditions for the boundedness of the
%    Hardy-Littlewood maximal operator between weighted $L^p$-spaces with
%    different weights},
%    journal={Proc. London Math. Soc. (3)},
%    volume={71},
%    date={1995},
%    number={1},
%    pages={135--157}
% }

% \bib{Saw}{article}{
%    author={Sawyer, Eric T.},
%    title={A characterization of two weight norm inequalities for fractional
%    and Poisson integrals},
%    journal={Trans. Amer. Math. Soc.},
%    volume={308},
%    date={1988},
%    number={2},
%    pages={533--545}
% }

\bib{TreVol2016}{article}{
   author={Treil, Sergei},
   author={Volberg, Alexander},
   title={Entropy conditions in two weight inequalities for singular
   integral operators},
   journal={Adv. Math.},
   volume={301},
   date={2016},
   pages={499--548}
}

\bib{Wil1987}{article}{
   author={Wilson, J. Michael},
   title={Weighted inequalities for the dyadic square function without
   dyadic $A_\infty$},
   journal={Duke Math. J.},
   volume={55},
   date={1987},
   number={1},
   pages={19--50},
   issn={0012-7094},
   review={\MR{883661}},
   doi={10.1215/S0012-7094-87-05502-5},
}		



\end{biblist}
\end{bibdiv}
\end{document}